\documentclass[final,nomarks]{dmtcs-episciences}
\usepackage{amsmath}
\usepackage[utf8]{inputenc}
\usepackage{subfigure}
\usepackage{latexsym}
\usepackage{amsfonts}
\usepackage{amssymb}
\usepackage{amsthm}
\usepackage{enumerate}
\usepackage{caption}
\usepackage{tabu}
\usepackage{float}
\usepackage{mathrsfs}

\usepackage{graphicx}
\usepackage{epstopdf}
\usepackage{epsfig}
\usepackage{caption}
 
\usepackage{bm} 
\usepackage{cite}

\makeatletter
\newtheorem*{rep@theorem}{\rep@title}
\newcommand{\newreptheorem}[2]{%
\newenvironment{rep#1}[1]{%
 \def\rep@title{#2 \ref{##1}}%
 \begin{rep@theorem}}%
 {\end{rep@theorem}}}
\makeatother

\newtheorem{theorem}{Theorem}[section]
\newreptheorem{theorem}{Conjecture}
\newtheorem{lemma}{Lemma}[section]

\newtheorem{corollary}{Corollary}[section]

\theoremstyle{definition}
\newtheorem{definition}{Definition}[section]
\newtheorem{remark}{Remark}[section]

\DeclareMathOperator{\des}{des}

\DeclareMathOperator{\peak}{peak}
\DeclareMathOperator{\tl}{tl}
\DeclareMathOperator{\tls}{tls}

\DeclareMathOperator{\zeil}{zeil}
\DeclareMathOperator{\rmax}{rmax}

\newcommand{\Comp}{\operatorname{Comp}}
\newcommand{\Av}{\operatorname{Av}}
\DeclareMathOperator{\SW}{\mathsf{SW}}
\DeclareMathOperator{\Bubb}{{\bf B}}

\author{Colin Defant\affiliationmark{1,2}\thanks{The author was supported by a Fannie and John Hertz Foundation Fellowship and an NSF Graduate Research Fellowship.}}
\title[Enumeration of Stack-Sorting Preimages]{Enumeration of Stack-Sorting Preimages via a Decomposition Lemma}

\affiliation{
  Princeton University}
\keywords{Permutation pattern; stack-sorting; bubble sort; Pop-stack-sorting; unbalanced Wilf equivalence.}

\received{2020-8-12}
\accepted{2021-03-23}

\begin{document}
\publicationdetails{22}{2021}{2}{3}{6709}
\maketitle

\begin{abstract}
We give three applications of a recently-proven ``Decomposition Lemma,'' which allows one to count preimages of certain sets of permutations under West's stack-sorting map $s$. We first enumerate the permutation class \[s^{-1}(\Av(231,321))=\Av(2341,3241,45231),\] finding a new example of an unbalanced Wilf equivalence. This result is equivalent to the enumeration of permutations sortable by $\Bubb\circ s$, where $\Bubb$ is the bubble sort map. We then prove that the sets $s^{-1}(\Av(231,312))$, $s^{-1}(\Av(132,231))=\Av(2341,1342,\underline{32}41,\underline{31}42)$, and $s^{-1}(\Av(132,312))=\Av(1342,3142,3412,34\underline{21})$ are counted by the so-called ``Boolean-Catalan numbers," settling a conjecture of the current author and another conjecture of Hossain. This completes the enumerations of all sets of the form $s^{-1}(\Av(\tau^{(1)},\ldots,\tau^{(r)}))$ for $\{\tau^{(1)},\ldots,\tau^{(r)}\}$ $\subseteq S_3$ with the exception of the set $\{321\}$. We also find an explicit formula for $|s^{-1}(\Av_{n,k}(231,312,321))|$, where $\Av_{n,k}(231,312,321)$ is the set of permutations in $\Av_n(231,312,321)$ with $k$ descents. This allows us to prove a conjectured identity involving Catalan numbers and order ideals in Young's lattice. 
\end{abstract}

\section{Introduction}\label{Sec:Intro}

\subsection{Background}
We use the word ``permutation" to refer to an ordering of a set of positive integers written in one-line notation. Let $S_n$ denote the set of permutations of the set $[n]$. If $\pi$ is a permutation of length $n$, then the \emph{standardization} of $\pi$ is the permutation in $S_n$ obtained by replacing the $i^\text{th}$-smallest entry in $\pi$ with $i$ for all $i$. We say a permutation is \emph{standardized} if it is equal to its standardization.

\begin{definition}\label{Def2}
Given $\tau\in S_m$, we say a permutation $\sigma=\sigma_1\cdots\sigma_n$ \emph{contains the pattern} $\tau$ if there exist indices $i_1<\cdots<i_m$ in $[n]$ such that the standardization of $\sigma_{i_1}\cdots\sigma_{i_m}$ is $\tau$. We say $\sigma$ \emph{avoids} $\tau$ if it does not contain $\tau$. Let $\Av(\tau^{(1)},\tau^{(2)},\ldots)$ denote the set of standardized permutations that avoid the patterns $\tau^{(1)},\tau^{(2)},\ldots$ (this list of patterns could be finite or infinite). A set of the form $\Av(\tau^{(1)},\tau^{(2)},\ldots)$ is called a \emph{permutation class}. Let $\Av_n(\tau^{(1)},\tau^{(2)},\ldots)=\Av(\tau^{(1)},\tau^{(2)},\ldots)\cap S_n$.  
\end{definition}

The investigation of permutation patterns began with Knuth's analysis of a certain ``stack-sorting algorithm" \cite{Knuth}. This analysis also marked the first use of what is now called the ``kernel method," a powerful enumerative tool that we employ throughout this article (see \cite{Bousquet4,Bousquet3,Banderier,
Prodinger} for more information about this technique). In his dissertation, West \cite{West} defined a deterministic variant of Knuth's algorithm. This variant is a function, which we call the ``stack-sorting map" and denote by $s$, that sends permutations to permutations. The map $s$ sends the empty permutation to itself. To define $s(\pi)$ when $\pi$ is a nonempty permutation, we write $\pi=LnR$, where $n$ is the largest entry in $\pi$. We then let $s(\pi)=s(L)s(R)n$. For example, \[s(14253)=s(142)\,s(3)\,5=s(1)\,s(2)\,4\,3\,5=12435.\] There is now a large amount of literature concerning the stack-sorting map \cite{Bona, BonaWords, BonaSurvey, BonaSimplicial, BonaSymmetry, Bousquet98, Bousquet, Bouvel, BrandenActions, Branden3, Claesson, Cori, DefantCatalan, 
DefantCounting,
DefantFertilityWilf, DefantPostorder, DefantPreimages, DefantClass, DefantMonotonicity, DefantTroupes, DefantEngenMiller, Dulucq, Dulucq2, Fang, Goulden, Hanna, Singhal, Ulfarsson, West, Zeilberger}. 

We say a permutation $\pi$ is $t$-\emph{stack-sortable} if $s^t(\pi)$ is an increasing permutation, where $s^t$ denotes the $t$-fold iterate of $s$. Let $\mathcal W_t(n)$ be the set of $t$-stack-sortable permutations in $S_n$, and let $W_t(n)=|\mathcal W_t(n)|$. Knuth initiated the study of pattern avoidance and the study of stack-sorting (and introduced the kernel method) with the following theorem. 

\begin{theorem}[\!\!\cite{Knuth}]\label{Thm1}
A permutation is $1$-stack-sortable if and only if it avoids the pattern $231$. Furthermore, $W_1(n)=|\Av_n(231)|=C_n$, where $C_n=\frac{1}{n+1}{2n\choose n}$ is the $n^\text{th}$ Catalan number. 
\end{theorem} 

In his dissertation, West conjectured that $W_2(n)=\frac{2}{(n+1)(2n+1)}{3n\choose n}$. Zeilberger \cite{Zeilberger} later proved this formula, and other proofs have emerged over the past few decades \cite{Cori,DefantCounting,Dulucq,Dulucq2,
Fang,Goulden}. Some authors have investigated the enumeration of $2$-stack-sortable permutations according to various statistics \cite{BonaSimplicial,Bousquet98,Bouvel,
DefantCounting,Dulucq}. There is little known about $t$-stack-sortable permutations when $t\geq 3$. The best known upper bound in general is given by the estimate $W_t(n)\leq (t+1)^{2n}$. The current author \cite{DefantPreimages} proved that $\lim\limits_{n\to\infty}W_3(n)^{1/n}<12.53296$ and $\lim\limits_{n\to\infty}W_4(n)^{1/n}<21.97225$. Recently, B\'ona \cite{BonaWords} reproved the first of these two estimates. 

Even more recently, the current author \cite{DefantCounting} proved a certain ``Decomposition Lemma" and used it to give a new proof of Zeilberger's formula for $W_2(n)$ that generalizes in order to count $2$-stack-sortable permutations according to their number of descents and number of peaks. The proof also generalizes to the setting of $3$-stack-sortable permutations, yielding a recurrence that generates the numbers $W_3(n)$ (more generally, it gives a recurrence for $W_3(n,k,p)$, the number of $3$-stack-sortable permutations of length $n$ with $k$ descents and $p$ peaks). This allowed the author to shed light on several conjectures of B\'ona concerning $3$-stack-sortable permutations, disproving one of them. 

West \cite{West} defined the \emph{fertility} of a permutation $\pi$ to be $|s^{-1}(\pi)|$, the number of preimages of $\pi$ under $s$. He then went through a great deal of effort to compute the fertilities of the permutations of the forms \[23\cdots k1(k+1)\cdots n,\quad 12\cdots(k-2)k(k-1)(k+1)\cdots n,\quad\text{and}\quad k12\cdots(k-1)(k+1)\cdots n.\] The fact that these permutations are of such specific forms indicates the difficulty of computing fertilities. Bousquet-M\'elou \cite{Bousquet} found a method for determining whether or not a given permutation is \emph{sorted}, meaning that its fertility is positive. She asked for a general method for computing the fertility of any given permutation. The current author achieved this in even greater generality in \cite{DefantPostorder, DefantPreimages, DefantClass} using new combinatorial objects called ``valid hook configurations." He and others have applied this method in order to prove several results about fertilities of permutations, many of which link the stack-sorting map to other fascinating combinatorial objects and sequences \cite{DefantCatalan, 
DefantCounting, 
DefantFertilityWilf, DefantMonotonicity, DefantTroupes, Hanna, 
DefantPostorder, DefantPreimages, DefantClass, DefantEngenMiller, Singhal}. 

Define the fertility of a set $A$ of permutations to be $|s^{-1}(A)|$. In \cite{DefantClass}, the current author computed the fertilities of many sets of the form $\Av_n(\tau^{(1)},\ldots,\tau^{(r)})$ for $\tau^{(1)},\ldots,\tau^{(r)}\in S_3$. He also refined these enumerative results according to the statistics that count descents and peaks. Investigations of preimages of permutation classes appeared earlier in \cite{Bouvel} and \cite{Claesson} and appeared more recently in \cite{DefantFertilityWilf}; this line of work is motivated by the observations that $\mathcal W_1(n)=s^{-1}(\Av_n(21))$ and $\mathcal W_2(n)=s^{-1}(\Av_n(231))$. Further motivation comes from the fact that preimages of permutation classes under $s$ are often themselves permutation classes. For example, $s^{-1}(\Av(321))=\Av(34251, 35241, 45231)$. Our goal in the present article is to compute fertilities of some sets of permutations that were not completed in \cite{DefantClass}. We do this via the Decomposition Lemma, which, when combined with generating function tools such as the kernel method, gives a unified technique for computing fertilities.  

\subsection{Summary of Main Results}
Section \ref{Sec:DecompLemma} is brief; its purpose is simply to state the Decomposition Lemma. In Section \ref{Sec:Unbalanced}, we use the Decomposition Lemma to prove that \[\sum_{n\geq 0}|s^{-1}(\Av_n(231,321))|x^n=\frac{1}{1-xC(xC(x))},\] where $C(x)=\dfrac{1-\sqrt{1-4x}}{2x}$ is the generating function of the Catalan numbers. This proves Conjecture~12.3 from \cite{DefantClass}. This is notable because \[s^{-1}(\Av(231,321))=\Av(2341,3241,45231).\] Callan \cite{Callan} proved that $\dfrac{1}{1-xC(xC(x))}$ is the generating function of the permutation class \linebreak $\Av(4321,4213)$, and Bloom and Vatter \cite{Bloom} later reproved this result with a clearer bijection. The authors of \cite{Mansour} showed that this expression is also the generating function for $\Av(2413,2431,23154)$, which means that $\Av(4321,4213)$ and $\Av(2413,2431,23154)$ form a pair of ``unbalanced" Wilf equivalent permutation classes. The first naturally-occurring unbalanced Wilf equivalent pairs of finitely-based\footnote{We say a permutation class is \emph{finitely-based} if it is of the form $\Av(\tau^{(1)},\ldots,\tau^{(r)})$, where the list $\tau^{(1)},\ldots,\tau^{(r)}$ is finite.} permutation classes were found only recently \cite{Bloom2, Burstein}. Our theorem proves that $\Av(2341,3241,45231)$ is also Wilf equivalent to $\Av(4321,4213)$ and $\Av(2413,2431,23154)$, providing yet another example of a pair of unbalanced Wilf equivalent finitely-based classes. 

Another motivation for the results of Section \ref{Sec:Unbalanced} comes from the study of the \emph{bubble sort map} $\Bubb$. This map sends the empty permutation to itself. If $\pi=LnR$, where $n$ is the largest entry of $\pi$, then $\Bubb(\pi)=\Bubb(L)Rn$. The article \cite{Albert} investigates permutations sortable via $s\circ\Bubb$, showing that \[(s\circ\Bubb)^{-1}(123\cdots n)=\Av_n(3241, 2341, 4231, 2431).\] Therefore, it is natural to ask about permutations that are sortable by $\Bubb\circ s$. It is known (see \cite{Albert}) that $\Bubb^{-1}(123\cdots n)=\Av_n(231,321)$, so it follows from the main result of Section \ref{Sec:Unbalanced} that \[(\Bubb\circ s)^{-1}(123\cdots n)=\Av_n(2341,3241,45231)\] and that \[\sum_{n\geq 0}|(\Bubb\circ s)^{-1}(123\cdots n)|x^n=\frac{1}{1-xC(xC(x))}.\]   

In Section \ref{Sec:Boolean}, we consider the chain of equalities \[\sum_{n\geq 1}|s^{-1}(\Av_n(132,312))|x^n=\sum_{n\geq 1}|s^{-1}(\Av_n(231,312))|x^n=\sum_{n\geq 1}|s^{-1}(\Av_n(132,231))|x^n\] \[=\frac{1-2x-\sqrt{1-4x-4x^2}}{4x}.\] The first of these equalities was proven in \cite{DefantClass}, while the second was proven in \cite{DefantFertilityWilf}. We use the Decomposition Lemma to prove the third equality, completing the chain and proving a conjecture from \cite{DefantClass}. Let us also remark that $s^{-1}(\Av(231,312))$ is the set of permutations that are sortable via one iteration of the map $\mathsf{Pop}\circ s$, where $\mathsf{Pop}$ is the pop-stack-sorting map (see \cite{Asinowski} and the references therein). Therefore, this theorem answers a deterministic variant of one of the problems considered in \cite{Smith}. 

The expression $\displaystyle \frac{1-2x-\sqrt{1-4x-4x^2}}{4x}$ is the generating function of OEIS sequence A071356 \cite{OEIS}. Hossain \cite{Hossain} has named these numbers ``Boolean-Catalan numbers" and has discussed several interesting combinatorial properties that they enjoy. The above chain of inequalities is important because one can show that \[s^{-1}(\Av(132,231))=\Av(2341,1342,\underline{32}41,\underline{31}42)\] and \[ s^{-1}(\Av(132,312))=\Av(1342,3142,3412,34\underline{21}),\] where the underlines are used to represent so-called ``vincular patterns" (defined in Section \ref{Sec:Boolean}). Therefore, we are able to use the stack-sorting map as a tool in order to enumerate sets of permutations avoiding certain vincular patterns. This is particularly noteworthy in the case of the set $s^{-1}(\Av(132,231))=\Av(2341,1342,\underline{32}41,\underline{31}42)$. Indeed, Hossain \cite{HossainPrivate} \emph{conjectured} that this set of permutations is enumerated by the Boolean-Catalan numbers. As far as we are aware, there is no known proof of this fact that does not make use of the stack-sorting map. 

Applying the Decomposition Lemma yet again in Section \ref{Sec:Descents}, we prove that \[|s^{-1}(\Av_{n,k}(231,312,321))|=\frac{1}{n+1}{n-k-1\choose k}{2n-2k\choose n},\] where $\Av_{n,k}(231,312,321)$ is the set of permutations in $\Av_n(231,312,321)$ with $k$ descents. This yields a new proof of Theorem 6.2 from \cite{DefantClass}. It also allows us to prove Conjecture 6.1 from that same article, which is an identity involving Catalan numbers and orderings on integer compositions and partitions. 

There are $64$ subsets $\{\tau^{(1)},\ldots,\tau^{(r)}\}$ of $S_3$; for each one, we have the problem of enumerating the permutations in $s^{-1}(\Av(\tau^{(1)},\ldots,\tau^{(r)}))$. Many of these enumerations were accomplished in \cite{Bouvel,DefantClass,Knuth,Zeilberger}. The main results of Sections \ref{Sec:Unbalanced} and \ref{Sec:Boolean} below finish the last remaining cases with the single exception of the set $\{321\}$. 

\begin{remark}
The permutation classes $\Av(231,321)$, $\Av(132,231)$, and $\Av(231,312,321)$ have very simple descriptions, so one might think that the enumerations of their preimages under $s$ could be computed directly and easily without additional tools. This is not the case. As mentioned above, even computing the fertilities of extremely simple permutations like $23\cdots k1(k+1)\cdots n$ takes some work without the Decomposition Lemma or valid hook configurations. Even with the machinery of valid hook configurations, the author could not figure out how to compute these fertilities in the paper \cite{DefantClass}. 
\end{remark}

\subsection{General Proof Strategy}
The \emph{tail length} of a permutation $\pi=\pi_1\cdots\pi_n\in S_n$, denoted $\tl(\pi)$, is the largest integer $\ell\in\{0,\ldots,n\}$ such that $\pi_i=i$ for all $i\in \{n-\ell+1,\ldots,n\}$. Our strategy for proving the results in Sections \ref{Sec:Unbalanced}--\ref{Sec:Descents} is to use the Decomposition Lemma in order to count various sets of permutations according to the statistic $\tls$ defined by $\tls(\pi)=\tl(s(\pi))$. One interesting property of this statistic is that it is closely related to the statistic $\zeil$ that Zeilberger used when he counted $2$-stack-sortable permutations. For $\sigma\in S_n$, $\zeil(\sigma)$ is the largest positive integer $m$ such that $n,n-1,\ldots,n-m+1$ appear in decreasing order in $\sigma$. It is proven in \cite{DefantFertilityWilf} that \[\zeil(\sigma)=\min\{\rmax(\sigma),\tls(\sigma)\},\] where $\rmax(\sigma)$ is the number of right-to-left maxima of $\sigma$. Our enumerations according to the statistic $\tls$ will yield generating function equations involving catalytic variables that we can remove using the kernel method.

\section{The Decomposition Lemma}\label{Sec:DecompLemma}
A \emph{descent} of a permutation $\pi=\pi_1\cdots\pi_n$ is an index $i\in[n-1]$ such that $\pi_i>\pi_{i+1}$. The \emph{plot} of $\pi$ is the diagram showing the points $(i,\pi_i)$ for all $i\in[n]$. For example, the image on the left in Figure \ref{Fig2} is the plot of $31542678$. A \emph{hook} of a permutation is obtained by starting at a point $(i,\pi_i)$ in the plot of $\pi$, drawing a vertical line segment moving upward, and then drawing a horizontal line segment to the right that connects with a point $(j,\pi_j)$. This only makes sense if $i<j$ and $\pi_i<\pi_j$. The point $(i,\pi_i)$ is called the \emph{southwest endpoint} of the hook, while $(j,\pi_j)$ is called the \emph{northeast endpoint}. Let $\SW_i(\pi)$ be the set of hooks of $\pi$ with southwest endpoint $(i,\pi_i)$. The right image in Figure \ref{Fig2} shows a hook of $31542678$. This hook is in $\SW_3(31542678)$ because its southwest endpoint is $(3,5)$.   

\begin{figure}[h]
  \begin{center}{\includegraphics[width=0.7\textwidth]{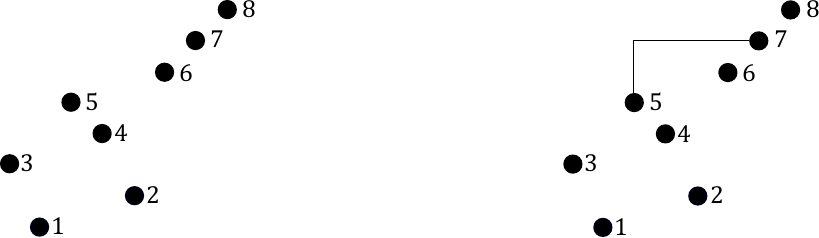}}
  \end{center}
  \caption{The left image is the plot of $31542678$. The right image shows this plot along with a single hook.}\label{Fig2}
\end{figure}

As mentioned above, the tail length $\tl(\pi)$ of a permutation $\pi=\pi_1\cdots\pi_n\in S_n$ is the largest integer $\ell\in\{0,\ldots,n\}$ such that $\pi_i=i$ for all $i\in \{n-\ell+1,\ldots,n\}$. The \emph{tail} of $\pi$ is the sequence of points $(n-\tl(\pi)+1,n-\tl(\pi)+1),\ldots,(n,n)$ in the plot of $\pi$. For example, the tail length of the permutation $31542678$ shown in Figure \ref{Fig2} is $3$, and the tail of this permutation is $(6,6),(7,7),(8,8)$. We say a descent $d$ of $\pi$ is \emph{tail-bound} if every hook in $\SW_d(\pi)$ has its northeast endpoint in the tail of $\pi$. The tail-bound descents of $31542678$ are $3$ and $4$.  

Suppose $H$ is a hook of a permutation $\pi=\pi_1\cdots\pi_n$ with southwest endpoint $(i,\pi_i)$ and northeast endpoint $(j,\pi_j)$. Let $\pi_U^H=\pi_1\cdots\pi_i\pi_{j+1}\cdots\pi_n$ and $\pi_S^H=\pi_{i+1}\cdots\pi_{j-1}$. The permutations $\pi_U^H$ and $\pi_S^H$ are called the \emph{$H$-unsheltered subpermutation of $\pi$} and the \emph{$H$-sheltered subpermutation of $\pi$}, respectively. For example, if $\pi=31542678$ and $H$ is the hook shown on the right in Figure \ref{Fig2}, then $\pi_U^H=3158$ and $\pi_S^H=426$. Note that the northeast endpoint of the hook does not contribute to either the unsheltered subpermutation nor the sheltered subpermutation. In all of the cases we consider in this paper, the plot of $\pi_S^H$ lies completely below the hook $H$ in the plot of $\pi$ (it is ``sheltered" by the hook $H$). 

\begin{lemma}[Decomposition Lemma, \cite{DefantCounting}]\label{Lem1}
If $d$ is a tail-bound descent of a permutation $\pi\in S_n$, then \[|s^{-1}(\pi)|=\sum_{H\in\SW_d(\pi)}|s^{-1}(\pi_U^H)|\cdot|s^{-1}(\pi_S^H)|.\]
\end{lemma}

\section{Enumerating $s^{-1}(\Av(231,321))$}\label{Sec:Unbalanced}

In his dissertation, West \cite{West} proved that a permutation is in $s^{-1}(\Av(231))$ (that is, it is $2$-stack-sortable) if and only if it avoids the pattern $2341$ and avoids any $3241$ pattern that is not part of a $35241$ pattern. As mentioned in the introduction, $s^{-1}(\Av(321))=\Av(34251,35241,45231)$. Combining these two facts, it is straightforward to check that 
\begin{equation}\label{Eq24}
s^{-1}(\Av(231,321))=\Av(2341,3241,45231).
\end{equation} In this section, we use the Decomposition Lemma to prove the following theorem. 
\begin{theorem}\label{Thm5}
We have \[\sum_{n\geq 0}|s^{-1}(\Av_n(231,321))|x^n=\frac{1}{1-xC(xC(x))},\] where $C(x)=\dfrac{1-\sqrt{1-4x}}{2x}$ is the generating function of the sequence of Catalan numbers. 
\end{theorem}
This is notable because it tells us that we can use the stack-sorting map and the Decomposition Lemma as tools in order to enumerate the permutation class $\Av(2341,3241,45231)$, which, as far as we are aware, has not been enumerated yet. As mentioned in the introduction, $\dfrac{1}{1-xC(xC(x))}$ is also the generating function of the sequence that enumerates the permutation classes $\Av(4321,4213)$ \cite{Bloom, Callan} and $\Av(2413,2431,23145)$ \cite{Mansour}. In particular, Theorem \ref{Thm5} tells us that the permutation classes \[\Av(4321,4213)\quad\text{and}\quad \Av(2341,3241,45231)\] form a new example of an unbalanced Wilf equivalence, as defined in \cite{Bloom2, Burstein}. 

\begin{proof}[of Theorem \ref{Thm5}]
Let \[\mathcal D_\ell(n)=\{\pi\in\Av_{n+\ell}(231,321):\tl(\pi)=\ell\}\quad\text{and}\quad\mathcal D_{\geq\ell}(n)=\{\pi\in\Av_{n+\ell}(231,321):\tl(\pi)\geq\ell\}.\] Let $B_\ell(n)=|s^{-1}(\mathcal D_\ell(n))|$ and $B_{\geq\ell}(n)=|s^{-1}(\mathcal D_{\geq\ell}(n))|$.

Suppose $\pi\in\mathcal D_\ell(n+1)$ is such that $\pi_{n+1-i}=n+1$ (see Figure \ref{Fig3}). Note that $n+1-i$ is a tail-bound descent of $\pi$ because every point in the plot of $\pi$ that is higher than $(n+1-i,n+1)$ is in the tail of $\pi$. The Decomposition Lemma tells us that $|s^{-1}(\pi)|$ is equal to the number of triples $(H,\mu,\lambda)$, where $H\in\SW_{n+1-i}(\pi)$, $\mu\in s^{-1}(\pi_U^H)$, and $\lambda\in s^{-1}(\pi_S^H)$. Choosing $H$ amounts to choosing the number $j\in\{1,\ldots,\ell\}$ such that the northeast endpoint of $H$ is $(n+1+j,n+1+j)$. The permutation $\pi$ and the choice of $H$ determine the permutations $\pi_U^H$ and $\pi_S^H$. On the other hand, the choices of $H$ and the permutations $\pi_U^H$ and $\pi_S^H$ uniquely determine $\pi$. It follows that $B_\ell(n+1)$, which is the number of ways to choose an element of $s^{-1}(\mathcal D_\ell(n+1))$, is also the number of ways to choose $j$, the permutations $\pi_U^H$ and $\pi_S^H$, and the permutations $\mu$ and $\lambda$. Let us fix a choice of $j$. 

\begin{figure}[h]
 \begin{center} {\includegraphics[width=0.43\textwidth]{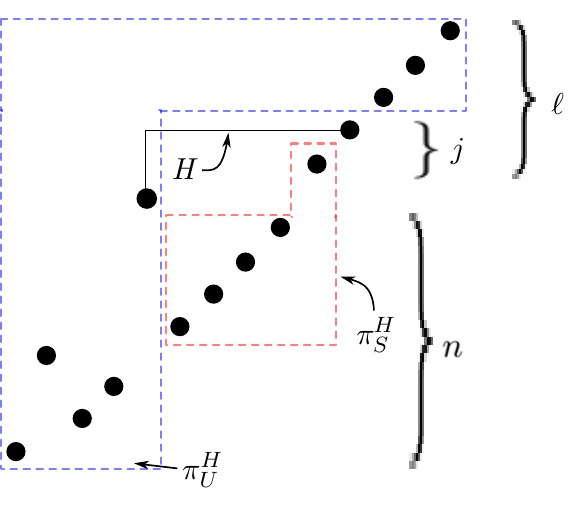}}
 \end{center}
  \caption{An example of a permutation $\pi\in D_{5}(9)$ and a hook $H$ with $i=4$ and $j=2$ (in the notation of the proof of Theorem \ref{Thm5}).}\label{Fig3}
\end{figure} 

Because $\pi$ avoids $231$, $\pi_U^H$ must be a permutation of the set \[\{1,\ldots,n-i\}\cup\{n+1\}\cup\{n+2+j,\ldots,n+\ell+1\},\] while $\pi_S^H$ must be a permutation of $\{n-i+1,\ldots,n+j\}\setminus\{n+1\}$. In fact, $\pi_S^H$ is the increasing permutation of the $(i+j-1)$-element set $\{n-i+1,\ldots,n+j\}\setminus\{n+1\}$ because $\pi$ avoids $321$. According to Theorem \ref{Thm1}, there are $C_{i+j-1}$ choices for $\lambda$. Choosing $\pi_U^H$ is equivalent to choosing its standardization, which is an element of $\mathcal D_{\geq \ell-j+1}(n-i)$. Any element of $\mathcal D_{\geq \ell-j+1}(n-i)$ can be chosen as the standardization of $\pi_U^H$. Also, $\pi_U^H$ has the same fertility as its standardization. Combining these facts, we find that the number of choices for $\pi_U^H$ and $\mu$ is $|s^{-1}(\mathcal D_{\geq \ell-j+1}(n-i))|=B_{\geq \ell-j+1}(n-i)$. We obtain the recurrence relation 
\begin{equation}\label{Eq35}
B_\ell(n+1)=\sum_{i=1}^n\sum_{j=1}^\ell C_{i+j-1}B_{\geq \ell-j+1}(n-i).
\end{equation}
Let \[G_\ell(x)=\sum_{n\geq 0}B_{\geq\ell}(n)x^n\quad\text{and}\quad I(x,y)=\sum_{\ell\geq 0}G_\ell(x)y^\ell.\] Note that \[G_\ell(0)=B_{\geq \ell}(0)=|s^{-1}(123\cdots\ell)|=C_\ell\] by Theorem \ref{Thm1}. Our goal is to understand the generating function \[I(x,0)=G_0(x)=\sum_{n\geq 0}B_{\geq 0}(n)x^n=\sum_{n\geq 0}|s^{-1}(\Av_n(231,321))|x^n.\] 

The recurrence \eqref{Eq35} tells us that \[\sum_{n\geq 0}B_\ell(n+1)x^n=\sum_{j=1}^\ell\sum_{n\geq 0}\sum_{i=1}^n C_{i+j-1}B_{\geq \ell-j+1}(n-i)x^n=\sum_{j=1}^\ell\left(\sum_{i\geq 1}C_{i+j-1}x^i\right)G_{\ell-j+1}(x),\] so \[\sum_{\ell\geq 0}\sum_{n\geq 0}B_\ell(n+1)x^ny^\ell=\sum_{\ell\geq 0}\sum_{j=1}^\ell\left(\sum_{i\geq 1}C_{i+j-1}x^i\right)G_{\ell-j+1}(x)y^\ell\] 
\begin{equation}\label{Eq25}=\left(\sum_{j\geq 1}\sum_{i\geq 1}C_{i+j-1}x^iy^{j-1}\right)(I(x,y)-I(x,0))=x\frac{C(x)-C(y)}{x-y}(I(x,y)-I(x,0)).
\end{equation} 
On the other hand, \[\sum_{\ell\geq 0}\sum_{n\geq 0}B_\ell(n+1)x^ny^\ell=\sum_{\ell\geq 0}\sum_{n\geq 0}B_{\geq\ell}(n+1)x^ny^\ell-\sum_{\ell\geq 0}\sum_{n\geq 0}B_{\geq \ell+1}(n)x^ny^\ell\]
\begin{equation}\label{Eq5}
=\frac{1}{x}\sum_{\ell\geq 0}(G_\ell(x)-C_\ell)y^\ell-\frac{1}{y}\sum_{\ell\geq 0}G_{\ell+1}(x)y^{\ell+1}=\frac{I(x,y)-C(y)}{x}-\frac{I(x,y)-I(x,0)}{y}.
\end{equation} 
Combining \eqref{Eq25} and \eqref{Eq5} and rearranging terms, we get the equation 
\begin{equation}\label{Eq27}
x(I(x,y)-I(x,0))\left(x\frac{C(x)-C(y)}{x-y}-\frac{1}{x}+\frac{1}{y}\right)=I(x,0)-C(y).
\end{equation} 

We now employ the kernel method (see \cite{Bousquet4,Bousquet3,Banderier,
Prodinger} for more on this method). There is a unique power series $Y=Y(x)$ such that $Y(x)=x+O(x^2)$ and $x\dfrac{C(x)-C(Y)}{x-Y}-\dfrac{1}{x}+\dfrac{1}{Y}=0$. Substituting this into \eqref{Eq27}, we find that $I(x,0)=C(Y)$. For ease of notation, let $u=I(x,0)$ and $v=\dfrac{1}{1-xC(xC(x))}$. We wish to show that $u=v$. 

Combining the equation $x\dfrac{C(x)-C(Y)}{x-Y}-\dfrac{1}{x}+\dfrac{1}{Y}=0$ with the standard Catalan functional equation $YC(Y)^2+1-C(Y)=0$ yields \[u=C(Y)=C(x)+\frac{(x-Y)^2}{x^2Y}=C(x)+\frac{(x-\frac{u-1}{u^2})^2}{x^2\frac{u-1}{u^2}}.\] If we solve for $C(x)$ and use the functional equation $xC(x)^2+1-C(x)=0$, we obtain \[x\left(u-\frac{(x-\frac{u-1}{u^2})^2}{x^2\frac{u-1}{u^2}}\right)^2+1-\left(u-\frac{(x-\frac{u-1}{u^2})^2}{x^2\frac{u-1}{u^2}}\right)=0.\] We can simplify this last equation to find that $Q(u,x)\dfrac{1-u+xu^2}{x^3u^4(u-1)^2}=0$, where $Q$ is the polynomial given by 
\begin{equation}\label{Eq28}
Q(a,x)=1-3a+(3+2x)a^2+(-1-4x+2x^2)a^3+(2x-2x^2+x^3)a^4.
\end{equation} It is easy to see that $1-u+xu^2\neq 0$, so $Q(u,x)=0$. 

We have $1-\dfrac{1}{v}=xC(xC(x))$, so \[C(x)\left(1-\frac{1}{v}\right)^2=x^2C(x)C(xC(x))^2=x(C(xC(x))-1),\] where the last equality follows from the Catalan functional equation. This means that \[1-x-C(x)\left(1-\frac{1}{v}\right)^2=1-xC(xC(x))=\frac{1}{v}.\] It follows that $C(x)=\dfrac{1-\frac{1}{v}-x}{(1-\frac{1}{v})^2}$, so we can use the Catalan functional equation to find that \[x\left(\frac{1-\frac{1}{v}-x}{(1-\frac{1}{v})^2}\right)^2+1-\frac{1-\frac{1}{v}-x}{(1-\frac{1}{v})^2}=0.\] After simplifying this equation, we find that $Q(v,x)\frac{1}{(v-1)^4}=0$, where $Q$ is the polynomial in \eqref{Eq28}. Hence, $Q(v,x)=0$. 

We now know that $Q(u,x)=Q(v,x)=0$. There are four Laurent series $F(x)$ satisfying $Q(F(x),x)\!=0$, but only one satisfies $F(x)=1+x+O(x^2)$. This proves that $u=v$, as desired. 
\end{proof}

\begin{remark}\label{Rem1}
Now that we have determined $I(x,0)$, one could use \eqref{Eq27} to find the bivariate generating function $I(x,y)$, which counts the permutations in the set \[s^{-1}(\Av(231,321))=\Av(2341,3241,45231)\] according to length and an additional statistic. 
\end{remark}

As mentioned in the introduction, Theorem \ref{Thm5} is equivalent to the enumeration of permutations that are sortable via the map $\Bubb\circ s$, where $\Bubb$ is the bubble sort map. 

\section{Enumerating $s^{-1}(\Av(132,231))$}\label{Sec:Boolean}
Many analogues and generalizations of classical permutation patterns have emerged over the past few decades. One of the most notable notions is that of a ``vincular pattern," which appeared first in \cite{Babson} and has garnered a large amount of attention ever since. We refer the reader to the survey \cite{Steingrimsson} and the references therein for further background. A \emph{vincular pattern} is a permutation pattern in which some consecutive entries can be underlined. We say a permutation \emph{contains} a vincular pattern if it contains an occurrence of the permutation pattern in which underlined entries are consecutive. For example, saying that a permutation $\sigma=\sigma_1\cdots\sigma_n$ contains the vincular pattern $\underline{32}41$ means that there are indices $i_1<i_2<i_3<i_4$ such that $\sigma_{i_4}<\sigma_{i_2}<\sigma_{i_1}<\sigma_{i_3}$ and $i_2=i_1+1$. We say a permutation \emph{avoids} a vincular pattern $\tau$ if it does not contain $\tau$. Let $\Av(\tau^{(1)},\tau^{(2)},\ldots)$ be the set of standardized permutations avoiding the vincular patterns $\tau^{(1)},\tau^{(2)},\ldots$, and let $\Av_n(\tau^{(1)},\tau^{(2)},\ldots)=\Av(\tau^{(1)},\tau^{(2)},\ldots)\cap S_n$. 

There are many fascinating combinatorial properties of the OEIS sequence A071356 \cite{OEIS}; several are listed in \cite{Hossain}, where the numbers in this sequence are named ``Boolean-Catalan numbers." One can define this sequence via its generating function \[\frac{1-2x-\sqrt{1-4x-4x^2}}{4x}.\] Hossain \cite{HossainPrivate} has conjectured that the Boolean-Catalan numbers enumerate the permutations in \[\Av(2341,1342,\underline{32}41,\underline{31}42).\] We will see that this set is precisely $s^{-1}(\Av(132,231))$. We will then enumerate these permutations via the Decomposition Lemma, proving Hossain's conjecture. It is interesting that the only known proof of this conjecture makes heavy use of the stack-sorting map. 

The main theorem of this section is the following.
\begin{theorem}\label{Thm6}
We have \[\sum_{n\geq 1}|s^{-1}(\Av_n(132,312))|x^n=\sum_{n\geq 1}|s^{-1}(\Av_n(231,312))|x^n=\sum_{n\geq 1}|s^{-1}(\Av_n(132,231))|x^n\] \[=\frac{1-2x-\sqrt{1-4x-4x^2}}{4x}.\]
\end{theorem}

This theorem settles Conjecture 10.1 from \cite{DefantClass}. Before proving it, we show how to deduce from it the following corollary, part of which settles Hossain's conjecture. 

\begin{corollary}
We have \[\sum_{n\geq 1}|\Av_n(2341,1342,\underline{32}41,\underline{31}42)|x^n=\sum_{n\geq 1}|\Av_n(1342,3142,3412,34\underline{21})|x^n\] \[=\frac{1-2x-\sqrt{1-4x-4x^2}}{4x}.\]
\end{corollary}

\begin{proof}
This corollary will follow immediately from Theorem \ref{Thm6} if we can show that \[s^{-1}(\Av(132,231))=\Av(2341,1342,\underline{32}41,\underline{31}42)\] and \[ s^{-1}(\Av(132,312))=\Av(1342,3142,3412,34\underline{21}).\] We prove the first of these equalities, the proof of the second is similar. If $\sigma$ contains $2341$ or $\underline{32}41$, then it follows from West's characterization of $2$-stack-sortable permutations (mentioned at the start of Section \ref{Sec:Unbalanced}) that $s(\sigma)$ contains $231$. It is also straightforward to check that $s(\sigma)$ contains $132$ if $\sigma$ contains $1342$ or $\underline{31}42$. This proves the containment \[s^{-1}(\Av(132,231))\subseteq \Av(2341,1342,\underline{32}41,\underline{31}42).\] 

Now suppose $s(\sigma)$ contains either $132$ or $231$. We want to prove that $\sigma$ contains one of the vincular patterns $2341,1342,\underline{32}41,\underline{31}42$. We will assume that $\sigma$ avoids $2341,1342,\underline{32}41$ and prove that it contains $\underline{31}42$. Let us first assume $s(\sigma)$ contains $231$. Using West's characterization of $2$-stack-sortable permutations and the assumption that $\sigma$ avoids $2341$, we see that $\sigma$ contains a $3241$ pattern that is not part of a $35241$ pattern. Therefore, there are indices $i_1<i_2<i_3<i_4$ such that $\sigma_{i_4}<\sigma_{i_2}<\sigma_{i_1}<\sigma_{i_3}$ and such that no entry lying between $\sigma_{i_1}$ and $\sigma_{i_2}$ is larger than $\sigma_{i_3}$. We may assume that among all such choices for $i_1,i_2,i_3,i_4$, we have made a choice that minimizes $i_2-i_1$. Because $\sigma$ avoids $\underline{32}41$, $i_2\geq i_1+2$. Our choice of $i_1,i_2,i_3,i_4$ guarantees that $\sigma_{i_1+1}<\sigma_{i_3}$. One can check that the minimality of $i_2-i_1$ forces $\sigma_{i_1+1}<\sigma_{i_4}$. However, this means that the standardization of $\sigma_{i_1+1}\sigma_{i_2}\sigma_{i_3}\sigma_{i_4}$ is $1342$, contradicting our assumption that $\sigma$ avoids $1342$. This shows that $s(\sigma)$ avoids $231$, so it must contain $132$. 

One can show \cite{Claesson, DefantClass} that a permutation is in $s^{-1}(\Av(132))$ if and only if it avoids the pattern $1342$ and avoids any $3142$ pattern that is not part of either a $34152$ pattern or a $35142$ pattern. Since we are assuming that $s(\sigma)$ contains $132$ and that $\sigma$ avoids $1342$, $\sigma$ must contain a $3142$ pattern that is not part of either a $34152$ pattern or a $35142$ pattern. This means that there are indices $i_1<i_2<i_3<i_4$ such that $\sigma_{i_2}<\sigma_{i_4}<\sigma_{i_1}<\sigma_{i_3}$ and such that no entry lying between $\sigma_{i_1}$ and $\sigma_{i_2}$ is larger than $\sigma_{i_1}$. We may assume that among all such choices for $i_1,i_2,i_3,i_4$, we have made a choice that minimizes $i_2-i_1$. In fact, one can show that this minimality assumption forces $i_2-i_1=1$. Hence, $\sigma_{i_1}\sigma_{i_2}\sigma_{i_3}\sigma_{i_4}$ is an occurrence of the vincular pattern $\underline{31}42$ in $\sigma$.
\end{proof}

In the following proof, we recycle our notation from the proof of Theorem \ref{Thm5} in Section \ref{Sec:Unbalanced}. This is meant to elucidate the parallels between the proofs. 

\begin{proof}[of Theorem \ref{Thm6}]
The first and second equalities in Theorem \ref{Thm6} were proven in \cite{DefantClass} and \cite{DefantFertilityWilf}, respectively, so we need only prove the last equality. Let \[\mathcal D_\ell(n)=\{\pi\in\Av_{n+\ell}(132,231):\tl(\pi)=\ell\}\quad\text{and}\quad\mathcal D_{\geq\ell}(n)=\{\pi\in\Av_{n+\ell}(132,231):\tl(\pi)\geq\ell\}.\] Let $B_\ell(n)=|s^{-1}(\mathcal D_\ell(n))|$ and $B_{\geq\ell}(n)=|s^{-1}(\mathcal D_{\geq\ell}(n))|$.

Suppose $n\geq 1$ and $\pi\in\mathcal D_\ell(n+1)$ (see Figure \ref{Fig4}). Using the fact that $\pi$ avoids $132$ and $231$ and has tail length $\ell<n+\ell+1$, we easily find that $\pi_1=n+1$ and that $1$ is a tail-bound descent of $\pi$. 
The Decomposition Lemma tells us that $|s^{-1}(\pi)|$ is equal to the number of triples $(H,\mu,\lambda)$, where $H\in\SW_1(\pi)$, $\mu\in s^{-1}(\pi_U^H)$, and $\lambda\in s^{-1}(\pi_S^H)$. Choosing $H$ amounts to choosing the number $j\in\{1,\ldots,\ell\}$ such that the northeast endpoint of $H$ is $(n+1+j,n+1+j)$. The permutation $\pi$ and the choice of $H$ determine the permutations $\pi_U^H$ and $\pi_S^H$. On the other hand, the choices of $H$ and the permutations $\pi_U^H$ and $\pi_S^H$ uniquely determine $\pi$. It follows that $B_\ell(n+1)$, which is the number of ways to choose an element of $s^{-1}(\mathcal D_\ell(n+1))$, is also the number of ways to choose $j$, the permutations $\pi_U^H$ and $\pi_S^H$, and the permutations $\mu$ and $\lambda$. Let us fix a choice of $j$.  

\begin{figure}[h]
\begin{center}
  {\includegraphics[width=0.487\textwidth]{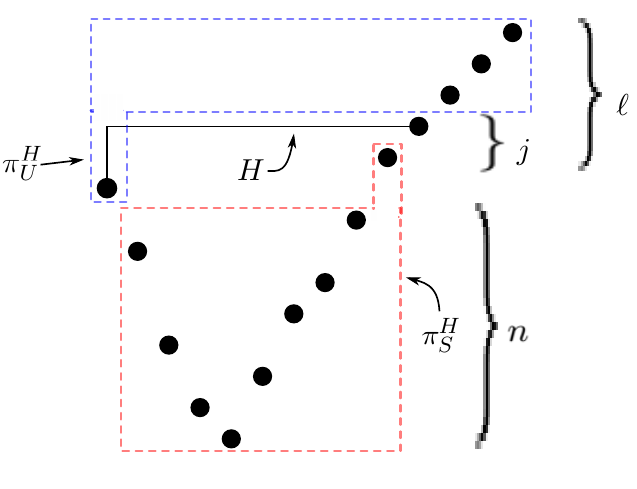}}\end{center}
  \caption{An example of a permutation $\pi\in D_{5}(9)$ and a hook $H$ with $j=2$ (in the notation of the proof of Theorem \ref{Thm6}).}\label{Fig4}
\end{figure} 

The permutation $\pi_U^H$ must be the increasing permutation $(n+1)(n+2+j)\cdots(n+\ell+1)$ of length $\ell-j+1$. By Theorem \ref{Thm1}, there are $C_{\ell-j+1}$ choices for $\mu$. Now, $\pi_S^H$ is a permutation of the $(n+j-1)$-element set $\{1,\ldots,n+j\}\setminus\{n+1\}$, so choosing $\pi_S^H$ is equivalent to choosing its standardization. This standardization is in $\mathcal D_{\geq j-1}(n)$. Any element of $\mathcal D_{\geq j-1}(n)$ can be chosen as the standardization of $\pi_S^H$. Also, $\pi_S^H$ has the same fertility as its standardization. Combining these facts, we find that the number of choices for $\pi_S^H$ and $\lambda$ is $|s^{-1}(\mathcal D_{\geq j-1}(n))|=B_{\geq j-1}(n)$. 
We obtain the recurrence relation 
\begin{equation}\label{Eq29}
B_\ell(n+1)=\sum_{j=1}^\ell C_{\ell-j+1}B_{\geq j-1}(n)\quad\text{for}\quad n\geq 1.
\end{equation} 
Let \[G_\ell(x)=\sum_{n\geq 0}B_{\geq\ell}(n)x^n\quad\text{and}\quad I(x,y)=\sum_{\ell\geq 0}G_\ell(x)y^\ell.\] Note that \[G_\ell(0)=B_{\geq \ell}(0)=|s^{-1}(123\cdots\ell)|=C_\ell\] by Theorem \ref{Thm1}. Let $C(x)=\sum_{n\geq 0}C_nx^n=\dfrac{1-\sqrt{1-4x}}{2x}$ be the generating function of the sequence of Catalan numbers. Since $B_{\geq 0}(n)=|s^{-1}(\Av_n(132,231))|$, our goal is to understand the generating function \[I(x,0)-1=G_0(x)-1=\sum_{n\geq 0}B_{\geq 0}(n)x^n-1=\sum_{n\geq 1}|s^{-1}(\Av_n(132,231))|x^n.\] 

Note that $B_\ell(1)=0$ because there are no permutations of length $\ell+1$ with tail length $\ell$. Therefore, the recurrence \eqref{Eq29} tells us that \[\sum_{n\geq 0}B_\ell(n+1)x^n=\sum_{n\geq 1}B_\ell(n+1)x^n=\sum_{j=1}^\ell C_{\ell-j+1}\sum_{n\geq 1}B_{\geq j-1}(n)x^n=\sum_{j=1}^\ell C_{\ell-j+1}(G_{j-1}(x)-C_{j-1}).\] Consequently, 
\begin{equation}\label{Eq30}
\sum_{\ell\geq 0}\sum_{n\geq 0}B_\ell(n+1)x^ny^\ell=\sum_{\ell\geq 0}\sum_{j=1}^\ell C_{\ell-j+1}(G_{j-1}(x)-C_{j-1})y^\ell=(C(y)-1)(I(x,y)-C(y)).
\end{equation}
On the other hand, \[\sum_{\ell\geq 0}\sum_{n\geq 0}B_\ell(n+1)x^ny^\ell=\sum_{\ell\geq 0}\sum_{n\geq 0}B_{\geq\ell}(n+1)x^ny^\ell-\sum_{\ell\geq 0}\sum_{n\geq 0}B_{\geq \ell+1}(n)x^ny^\ell\]
\begin{equation}\label{Eq31}
=\frac{1}{x}\sum_{\ell\geq 0}(G_\ell(x)-C_\ell)y^\ell-\frac{1}{y}\sum_{\ell\geq 0}G_{\ell+1}(x)y^{\ell+1}=\frac{I(x,y)-C(y)}{x}-\frac{I(x,y)-I(x,0)}{y}.
\end{equation} Combining \eqref{Eq30} and \eqref{Eq31} and rearranging terms, we get the equation 
\begin{equation}\label{Eq32}
I(x,y)(xy(C(y)-1)-y+x)=-yC(y)+xI(x,0)+xyC(y)(C(y)-1).
\end{equation} 

As in the previous section, we use the kernel method. Let \[Y=Y(x)=\frac{3x+2x^2-x\sqrt{1-4x-4x^2}}{2(1+x)^2}.\] After verifying that $xY(C(Y)-1)-Y+x=0$, we use \eqref{Eq32} to find that \[I(x,0)-1=-1+\frac{1}{x}(YC(Y))-xYC(Y)(C(Y)-1).\] After some elementary algebraic manipulations, this expression simplifies to the generating function \[\frac{1-2x-\sqrt{1-4x-4x^2}}{4x},\] as desired.
\end{proof}

\begin{remark}\label{Rem3}
Now that we have determined $I(x,0)$, one could use \eqref{Eq32} to find the bivariate generating function $I(x,y)$, which counts the permutations in the set \[s^{-1}(\Av(132,231))=\Av(2341,1342,\underline{32}41,\underline{31}42)\] according to length and an additional statistic. 
\end{remark}

\section{Counting $s^{-1}(\Av_{n,k}(231,312,321))$}\label{Sec:Descents}

A \emph{composition of $b$ into $a$ parts} is an $a$-tuple of positive integers that sum to $b$. Let $\Comp_a(b)$ denote the set of all compositions of $b$ into $a$ parts. There is a natural partial order $\preceq$ on $\Comp_a(b)$ defined by declaring that $(x_1,\ldots,x_a)\preceq(y_1,\ldots,y_a)$ if $\sum_{i=1}^\ell x_i\leq\sum_{i=1}^\ell y_i$ for all $\ell\in\{1,\ldots,a\}$. For $x=(x_1,\ldots,x_a)\in\Comp_a(b)$, let \[D_x=|\{y\in\Comp_a(b):y\preceq x\}|.\] Let $\psi(x)$ be the partition (composition with nonincreasing parts) that has $x_i-1$ parts of size $a-i$ for all $i\in\{1,\ldots,a-1\}$ (and has no parts of size at least $a$). The quantity $D_x$ can also be interpreted as the number of partitions whose Young diagrams fit inside the Young diagram of $\psi(x)$. Said differently, $D_x$ is the size of the order ideal in Young's lattice generated by $\psi(x)$ (see \cite{DefantClass} for more details). Let us write $C_x=\prod_{t=1}^a C_{x_t}$, where $C_j$ is the $j^\text{th}$ Catalan number. Put together, Theorems 6.1 and 6.2 in \cite{DefantClass} state that 
\begin{equation}\label{Eq15}
|s^{-1}(\Av_n(231,312,321))|=\sum_{k=0}^{n-1}\frac{1}{n+1}{n-k-1\choose k}{2n-2k\choose n}=\sum_{k=0}^{n-1}\sum_{q\,\in\,\Comp_{k+1}(n-k)}C_qD_q.
\end{equation}
From this equation, it is natural to conjecture that 
\begin{equation}\label{Eq16}
\sum_{q\,\in\,\Comp_{k+1}(n-k)}C_qD_q=\frac{1}{n+1}{n-k-1\choose k}{2n-2k\choose n}
\end{equation}
for all nonnegative integers $n$ and $k$; this is the content of Conjecture 6.1 in \cite{DefantClass}. In that article, the current author showed that this conjecture is equivalent to the following theorem. Let $\des(\pi)$ denote the number of descents of a permutation $\pi$. 
\begin{theorem}\label{Thm4}
For all $n\geq 1$ and $k\geq 0$, we have \[|s^{-1}(\Av_{n,k}(231,312,321))|=\frac{1}{n+1}{n-k-1\choose k}{2n-2k\choose n},\] where $\Av_{n,k}(231,312,321)=\{\pi\in\Av_n(231,312,321):\des(\pi)=k\}$. 
\end{theorem}
In this section, we use the Decomposition Lemma to prove this theorem, thereby settling the aforementioned conjecture. Our result implies the first equality in \eqref{Eq15}, so the proof given in this section yields a completely new proof of Theorem 6.2 from \cite{DefantClass} (see that article for more details).   

\begin{proof}[Proof of Theorem \ref{Thm4}]
Let \[\mathcal D_\ell(n,k)=\{\pi\in\Av_{n+\ell,k}(231,312,321):\tl(\pi)=\ell\}\] and \[\mathcal D_{\geq\ell}(n,k)=\{\pi\in\Av_{n+\ell,k}(231,312,321):\tl(\pi)\geq\ell\}.\] Let $B_\ell(n,k)=|s^{-1}(\mathcal D_\ell(n,k))|$ and $B_{\geq\ell}(n,k)=|s^{-1}(\mathcal D_{\geq\ell}(n,k))|$.

Suppose $n\geq 0$ and $\pi\in\mathcal D_\ell(n+1,k)$ (see Figure \ref{Fig5}). Using the fact that $\pi$ avoids $312$ and $321$, we easily find that $\pi_n=n+1$. Because $\pi$ avoids $231$, we must have $\pi_{n+1}=n$. Note that $n$ is a tail-bound descent of $\pi$. The Decomposition Lemma tells us that $|s^{-1}(\pi)|$ is equal to the number of triples $(H,\mu,\lambda)$, where $H\in\SW_{n}(\pi)$, $\mu\in s^{-1}(\pi_U^H)$, and $\lambda\in s^{-1}(\pi_S^H)$. Choosing $H$ amounts to choosing the number $j\in\{1,\ldots,\ell\}$ such that the northeast endpoint of $H$ is $(n+1+j,n+1+j)$. Let us fix such a $j$. The permutation $\pi_S^H=n(n+2)(n+3)\cdots(n+j)$ is an increasing permutation of length $j$. Therefore, there are $C_j$ choices for $\lambda$ by Theorem \ref{Thm1}. 

\begin{figure}[h]
\begin{center}  {\includegraphics[width=0.464\textwidth]{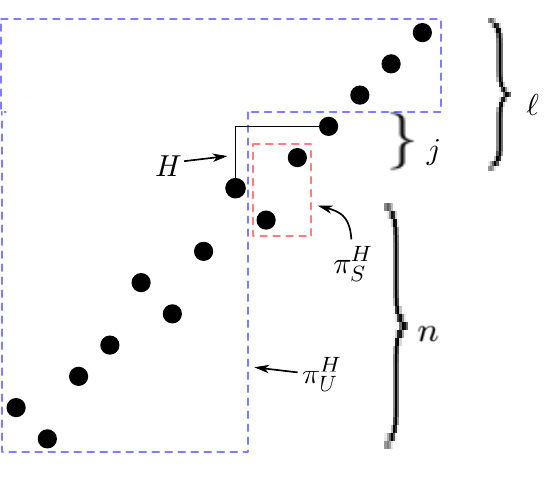}}
\end{center}
  \caption{An example of a permutation $\pi\in D_{5}(9,3)$ and a hook $H$ with $j=2$ (in the notation of the proof of Theorem \ref{Thm6}).}\label{Fig5}
\end{figure} 

Now, $\pi_U^H$ is a permutation of the set $\{1,\ldots,n-1\}\cup\{n+1\}\cup\{n+j+2,\ldots,n+\ell+1\}$, which has $n+\ell-j$ elements. Therefore, choosing $\pi_U^H$ is equivalent to choosing its standardization. The descents of $\pi_U^H$ are simply the descents of $\pi$ other than $n$, so $\des(\pi_U^H)=\des(\pi)-1=k-1$. It follows that the standardization of $\pi_U^H$ is in $\mathcal D_{\geq \ell-j+1}(n-1,k-1)$. Any permutation in $\mathcal D_{\geq \ell-j+1}(n-1,k-1)$ can be chosen as the standardization of $\pi_U^H$. Since $\pi_U^H$ has the same fertility as its standardization, the number of choices for $\pi_U^H$ and $\mu$ is $|s^{-1}(\mathcal D_{\geq \ell-j+1}(n-1,k-1))|=B_{\geq\ell-j+1}(n-1,k-1)$. We obtain the recurrence relation 
\begin{equation}\label{Eq18}
B_\ell(n+1,k)=\sum_{j=1}^\ell B_{\geq \ell-j+1}(n-1,k-1)C_j.
\end{equation}
Let \[G_{\ell,k}(x)=\sum_{n\geq 0}B_{\geq\ell}(n,k)x^n,\quad I_k(x,y)=\sum_{\ell\geq 0}G_{\ell,k}(x)y^\ell, \quad\text{and}\quad J(x,y,z)=\sum_{k\geq 0}I_k(x,y)z^k.\] Note that \[G_{\ell,k}(0)=B_{\geq \ell}(0,k)=|s^{-1}(\mathcal D_{\geq \ell}(0,k))|=\delta_{k0}C_\ell,\] where $\delta_{k0}$ is the Kronecker delta. Let $C(x)=\sum_{n\geq 0}C_nx^n=\dfrac{1-\sqrt{1-4x}}{2x}$. 
Because $B_{\geq 0}(n,k)=|s^{-1}(\Av_{n,k}(231,312,321))|$, our goal is to understand the generating function \[J(x,0,z)=\sum_{n\geq 0}\sum_{k\geq 0}B_{\geq 0}(n,k)x^nz^k=\sum_{n\geq 0}\sum_{k\geq 0}|s^{-1}(\Av_{n,k}(231,312,321))|x^nz^k.\]  

The recurrence \eqref{Eq18} is equivalent to \[\sum_{n\geq 0}B_\ell(n+1,k)x^n=x\sum_{j=1}^\ell G_{\ell-j+1,k-1}(x)C_j,\] so 
\begin{equation}\label{Eq19}
\sum_{\ell\geq 0}\sum_{n\geq 0}B_\ell(n+1,k)x^ny^\ell=x\sum_{\ell\geq 0}\sum_{j=1}^\ell G_{\ell-j+1,k-1}(x)C_jy^\ell=x\frac{C(y)-1}{y}(I_{k-1}(x,y)-I_{k-1}(x,0)).
\end{equation} 
On the other hand, \[\sum_{\ell\geq 0}\sum_{n\geq 0}B_\ell(n+1,k)x^ny^\ell=\sum_{\ell\geq 0}\sum_{n\geq 0}B_{\geq\ell}(n+1,k)x^ny^\ell-\sum_{\ell\geq 0}\sum_{n\geq 0}B_{\geq \ell+1}(n,k)x^ny^\ell\]
\begin{equation}\label{Eq20}
=\frac{1}{x}\sum_{\ell\geq 0}(G_{\ell,k}(x)-G_{\ell,k}(0))y^\ell-\frac{1}{y}\sum_{\ell\geq 0}G_{\ell+1,k}(x)y^{\ell+1}=\frac{I_k(x,y)-\delta_{k0}C(y)}{x}-\frac{I_k(x,y)-I_k(x,0)}{y}.
\end{equation} 
Combining \eqref{Eq19} and \eqref{Eq20} yields the equation 
\begin{equation}\label{Eq21}
x\frac{C(y)-1}{y}(I_{k-1}(x,y)-I_{k-1}(x,0))=\frac{I_k(x,y)-\delta_{k0}C(y)}{x}-\frac{I_k(x,y)-I_k(x,0)}{y}.
\end{equation} 
If we multiply both sides of \eqref{Eq21} by $z^k$ and sum over $k\geq 0$, we obtain \[xz\frac{C(y)-1}{y}(J(x,y,z)-J(x,0,z))=\frac{J(x,y,z)-C(y)}{x}-\frac{J(x,y,z)-J(x,0,z)}{y}.\] We can rewrite this equation as 
\begin{equation}\label{Eq22}
\frac{1}{y}(J(x,y,z)-J(x,0,z))(x^2z(C(y)-1)+x-y)=J(x,0,z)-C(y).
\end{equation}

As in the previous two sections, we use the kernel method to find the generating function $J(x,0,z)$. There is a unique power series $Y=Y(x,z)$ such that $C(Y(x,z))=x+2x^2+O(x^3)$ and $x^2z(C(Y)-1)+x-Y=0$. Substituting this into \eqref{Eq22} shows that $J(x,0,z)=C(Y)$. Let $\widehat J(x,z)=xJ(x,0,z)=xC(Y)$. Using the Catalan functional equation $YC(Y)^2+1-C(Y)=0$, we obtain 
\begin{equation}\label{Eq23}
\widehat J(x,z)-\frac{\widehat J(x,z)^2}{1-z\widehat J(x,z)^2}=x\left(C(Y)-\frac{xC(Y)^2}{1-x^2zC(Y)^2}\right)=x\left(C(Y)-\frac{x\frac{C(Y)-1}{Y}}{1-x^2z\frac{C(Y)-1}{Y}}\right).
\end{equation} We now use the fact that $x^2z(C(Y)-1)+x-Y=0$ to see that \[C(Y)-\frac{x\frac{C(Y)-1}{Y}}{1-x^2z\frac{C(Y)-1}{Y}}=1+\frac{Y-x}{x^2z}-\frac{x\frac{Y-x}{x^2zY}}{1-x^2z\frac{Y-x}{x^2zY}}=1.\] Substituting this into \eqref{Eq23} shows that \[\widehat J(x,z)=x+\frac{\widehat J(x,z)^2}{1-z\widehat J(x,z)^2}.\] By Lagrange inversion, \[\widehat J(x,z)=
x+\sum_{m\geq 0}\frac{1}{(m+1)!}\frac{\partial^m}{\partial x^m}\left(\frac{x^2}{1-zx^2}\right)^{m+1}\] \[=
x+\sum_{m\geq 0}\frac{1}{(m+1)!}\frac{\partial^m}{\partial x^m}\sum_{k\geq 0}{k+m\choose k}x^{2m+2k+2}z^k\] \[=x+\sum_{m\geq 0}\frac{1}{m+1}\sum_{k\geq 0}{k+m\choose k}{2m+2k+2\choose m+2k+2}x^{m+2k+2}z^k\] \[=x+x\sum_{k\geq 0}\sum_{n\geq 2k+1}\frac{1}{n+1}{n-k-1\choose k}{2n-2k\choose n}x^nz^k,\] as desired.  
\end{proof}

\section{Further Directions}

The proofs of our three main theorems relied on the Decomposition Lemma, which was proven in \cite{DefantCounting}. That article actually obtained this useful tool as a corollary of a more general result, called the ``Refined Decomposition Lemma," which allows one to count stack-sorting preimages of permutations according to the additional statistics $\des$ and $\peak$ (for $\pi=\pi_1\cdots\pi_n$, $\peak(\pi)$ is the number of indices $i\in\{2,\ldots,n-1\}$ such that $\pi_{i-1}<\pi_i>\pi_{i+1}$). It could be interesting to generalize Theorems \ref{Thm5}, \ref{Thm6}, and \ref{Thm4} by counting the relevant preimage sets according to the statistics $\des$ and $\peak$. One could also attempt to generalize these results to the context of ``troupes,'' which are special families of colored binary plane trees for which a version of the Decomposition Lemma holds (see \cite{DefantTroupes} for more details).  Of course, one could also search for other interesting sets of permutations whose stack-sorting preimages can be enumerated via the Decomposition Lemma.  

The results in Sections \ref{Sec:Unbalanced} and \ref{Sec:Boolean} complete the enumerations of all sets of the form \linebreak $s^{-1}(\Av(\tau^{(1)},\ldots,\tau^{(r)}))$ for $\{\tau^{(1)},\ldots,\tau^{(r)}\}\subseteq S_3$ except the set $\{321\}$. There are currently no nontrivial results that enumerate sets of this form when one or more of the patterns $\tau^{(i)}$ has length at least $4$. Obtaining such results would be a natural next step. It would be especially interesting to solve this enumerative problem using valid hook configurations (as in \cite{DefantClass}) or the Decomposition Lemma in the cases in which the sets $s^{-1}(\Av(\tau^{(1)},\ldots,\tau^{(r)}))$ have some natural alternative descriptions (i.e., if they are characterized as the permutations avoiding some collection of patterns or vincular patterns). It is shown in \cite{DefantFertilityWilf} that \[\sum_{n\geq 1}|s^{-1}(\Av_n(132,3412))|x^n=\sum_{n\geq 1}|s^{-1}(\Av_n(231,1423))|x^n=\sum_{n\geq 1}|s^{-1}(\Av_n(312,1342))|x^n.\] Thus, it is natural to consider the specific problem of computing the generating function \[\sum_{n\geq 1}|s^{-1}(\Av_n(132,3412))|x^n.\]

\section{Acknowledgments}
The author was supported by a Fannie and John Hertz Foundation Fellowship and an NSF Graduate Research Fellowship.

\end{document}